\documentclass[12pt,leqno]{amsart} 
\usepackage[numbers]{natbib}
\usepackage{amsmath,amsthm}
\usepackage{amssymb}
\renewcommand{\widehat}{\hat}
\usepackage{url}
\usepackage[colorlinks={true},
backref,pdfstartview={FitBH},pdfview={FitBH},bookmarks={true}]{hyperref}
\global\hbadness=500
\global\tolerance=200
\newtheorem{theorem}{Theorem}[section]
\newtheorem{lemma}[theorem]{Lemma}
\newtheorem{conjecture}[theorem]{Conjecture}

\newtheorem{question}[theorem]{Question}

\newtheorem{corollary}[theorem]{Corollary} 
\theoremstyle{definition}
\newtheorem{definition}[theorem]{Definition}

\theoremstyle{remark}

\numberwithin{equation}{theorem}

\renewcommand{\phi}{\varphi}

\newcommand{\E}{\mathcal{E}}

\newcommand{\D}{\mathcal{D}}

\newcommand{\Ahat}{\widehat{A}}

  \newcommand
{\wock}{\omega_1^{\textup{CK}}}

\title[The Computably Enumerable Sets]{Some recent research directions
  in the computably enumerable sets}
  
\author[P.\ Cholak]{Peter~A.~Cholak}

\address{Department of Mathematics\\ University of Notre Dame\\ 
  Notre Dame, IN 46556-5683}

\email{Peter.Cholak.1@nd.edu}

\urladdr{http://www.nd.edu/~cholak}


\begin{document}

 \begin{abstract}
   As suggested by the title, this paper is a survey of recent results
   and questions on the collection of computably enumerable sets under
   inclusion.  This is not a broad survey but one focused on the
   author's and a few others' current research.
 \end{abstract}
 \maketitle

 There are many equivalent ways to definite a computably enumerable or
 c.e.\ set.  The one that we prefer is the domain of a Turing machine
 or the set of balls accepted by a Turing machine. Perhaps this
 definition is the main reason that this paper is included in this
 volume and the corresponding talk in the ``Incomputable'' conference.
 The c.e.\ sets are also the sets which are $\Sigma^0_1$ definable in
 arithmetic.

 There is a computable or effective listing, $\{ M_e | e \in \omega\}$,
 of all Turing machines.  This gives us a listing of all c.e.\ sets,
 $x$ in $W_e$ at stage $s$ iff $M_e$ with input $x$ accepts by stage
 $s$.  This enumeration of all c.e.\ sets is very dynamic. We can
 think of balls $x$ as flowing from one c.e.\ set into another.  Since
 they are sets, we can partially order them by inclusion, $\subseteq$
 and consider them as model, $\mathcal{E} = \langle \{ W_e | e \in
 \omega\}, \subseteq\rangle$.  All sets (not just c.e.\ sets) are
 partially ordered by Turing reducibility, where $A \leq_T B$ iff
 there is a Turing machine that can compute $A$ given an oracle for
 $B$.

 Broadly, our goal is to study the structure $\mathcal{E}$ and learn
 what we can about the interactions between definability (in the
 language of inclusion $\subseteq$), the dynamic properties of c.e.\
 sets and their Turing degrees.  A very rich relationship between
 these three notions has been discovered over the years.  We cannot
 hope to completely cover this history in this short paper. But, we
 hope that we will cover enough of it to show the reader that the
 interplay between these three notions on c.e.\ sets is, and will
 continue to be, an very interesting subject of research.

 We are assuming that the reader has a background in computability
 theory as found in the first few chapters of \citet{Soare:87}.  All
 unknown notation also follows \cite{Soare:87}.

 \section{Friedberg Splits}

 The first result in this vein was \citet{Friedberg:58}, every
 noncomputable c.e.\ set has a Friedberg split.  Let us first
 understand the result then explore why we feel this result relates to
 the interplay of definability, Turing degrees and dynamic properties
 of c.e.\ sets.

\begin{definition}
  $A_0 \sqcup A_1 = A$ is a \emph{Friedberg split} of $A$ iff, for all
  $W$ (all sets in this paper are always c.e.), if $W-A$ is not a
  c.e.\ set neither are $W-A_i$.
\end{definition}

The following definition depends on the chosen enumeration of all
c.e. sets.  We use the enumeration given to us in the second paragraph
of this paper, $x \in W_{e,s}$ iff $M_e$ with input $x$ accepts by
stage $s$, but with the convention that if $x \in W_{e,s}$ then $e,x <
s$ and, for all stages $s$, there is at most one pair $e, x$ where $x$
enters $W_e$ at stage $s$.  Some details on how we can effectively
achieve this type of enumeration can be found in \citet[Exercise
I.3.11]{Soare:87}.  Moreover, when given a c.e.\ set, we are given the
index of this c.e.\ set in terms of our enumeration of all c.e.\
sets. At times we will have to appeal to Kleene's Recursion Theorem to
get this index.

\begin{definition}
  For c.e.\ sets $A = W_e$ and $B=W_i$, $$A \backslash B = \{ x |
  \exists s [x \in (W_{e,s}-W_{i,s})]\}$$ and $A\searrow B = A
  \backslash B \cap B$.
\end{definition}

By the above definition, $A\backslash B$ is a c.e.\ set.  $A
\backslash B$ is the set of balls that enter $A$ before they enter
$B$.  If $x \in A \backslash B$ then $x$ may or may not enter $B$ and
if $x$ does enters $B$, it only does so after $x$ enters $A$ (in terms
of our enumeration). Since the intersection of two c.e.\ sets is c.e,
$A \searrow B$ is a c.e.\ set.  $A\searrow B $ is the c.e.\ set of
balls $x$ that first enter $A$ and then enter $B$ (under the above
enumeration).

Note that $W\backslash A = (W-A) \sqcup (W\searrow A)$ ($\sqcup$ is
the disjoint union). Since $W\backslash A$ is a c.e.\ set, if $W-A$
is not a c.e.\ set then $W\searrow A$ must be infinite.  (This happens
for all enumerations.)  Hence infinitely many balls from $W$ must flow
into $A$.

\begin{lemma}[Friedberg]
  Assume $A= A_0 \sqcup A_1$, and, for all $e$, if $W_e \searrow A$ is
  infinite then both $W_e \searrow A_0$ and $W_e \searrow A_1$ are
  infinite.  Then $A_0 \sqcup A_1$ is a Friedberg split of $A$.
  Moreover if $A$ is not computable neither are $A_0$ and $A_1$.
\end{lemma}

\begin{proof}
  Assume that $W-A$ is not a c.e.\ set but $X= W-A_0$ is a c.e.\ set.
  $X-A = W-A$ is not a c.e.\ set.  So $X\searrow A$ is infinite and
  therefore $X \searrow A_0$ is infinite.  Contradiction.

  If $A_0$ is computable then $X= \overline{A}_0$ is a c.e.\ set and
  if $A$ is not computable then $X-A$ cannot be a c.e.\ set.  So use
  the same reasoning as above to show $X \searrow A_0$ is infinite for
  a contradiction.
\end{proof}

Friedberg more or less invented the priority method to split every
c.e.\ set into two disjoint c.e.\ sets while meeting the hypothesis of
the above lemma.  The main idea of Friedberg's construction is when a
ball $x$ enters $A$ at stage $s$ to add it to one of $A_0$ or $A_1$
but which set $x$ enters is determined by priority.  Let
\begin{equation}\tag*{$\mathcal{P}_{e,i,k}$}
  \label{eq:1}
  \text{if }W_e \searrow A \text{ is infinite then } |W_e \searrow A_i
  | \geq k. 
\end{equation}
We say $x$ meets $\mathcal{P}_{e,i,k}$ at stage $s$ if $|W_e \searrow
A_i | < k$ by stage $s-1$ and if we add $x$ to $A_i$ then $|W_e
\searrow A_i | \geq k$ by stage $s$.  Find the highest $\langle e, i,
k \rangle$ that $x$ can meet and add $x$ to $A_i$ at stage $s$.  It is
not hard to show that all the $\mathcal{P}_{e,i,k}$ are meet.

It is clear that the existence of a Friedberg split is very dynamic.
Let's see why it is also a definable property.  But, first, we need to
understand what we can say about $\mathcal{E}$ with inclusion.  We are
not going to go through the details but we can define union,
intersection, disjoint union, the empty set and the whole set.  We can
say that a set is complemented.  A very early result shows that if $A$
and $\overline{A}$ are both c.e.\ then $A$ is computable. So it is
definable if a c.e.\ set is
computable.  
Inside every computable set we can repeat the construction of the
halting set.  So a c.e.\ set $X$ is finite iff every subset of $X$ is
computable.  Hence $W-A$ is a c.e.\ set iff there is a c.e. set $X$
disjoint from $A$ such that $W \cup A = X \sqcup A$.  So saying that
$A_0 \sqcup A_1 =A$ is a Friedberg split and $A$ is not computable is
definable.

Friedberg's result answers a question of Myhill, ``Is every
non-recursive, recursively enumerable set the union of two disjoint
non-recursive,  recursively enumerable sets?''  The question of Myhill
was asked in print in the Journal of Symbolic Logic in June 1956,
Volume 21, Number 2 on page 215 in the ``Problems'' section of the
JSL.  This question was the eighth problem appearing in this section.
The question about the existence of maximal sets, also answered by
Friedberg, was ninth.  This author does not know how many questions
were asked or when this section was dropped. Myhill also reviewed
\citet{Friedberg:58} for the AMS, but the review left no clues why he
asked the question in the first place.

The big question in computability theory in the 1950's was ``Does
there exist an incomplete noncomputable c.e. set''?
\citet{Kleene.Post:54} showed that there are a pair of incomparable
Turing degrees below $\bf{0'}$.  We feel that after Kleene-Post,
Myhill's question is very natural.  So we can claim that the existence
of a Friedberg split for every c.e.\ set $A$ fits into our theme, the
interplay of definability, dynamic properties and Turing degree on the
c.e.\ sets.

\subsubsection{Recent Work and Questions on Friedberg Splits}

Given a c.e.\ set one can uniformly find a Friedberg split.  It is
known that there are other types of splits.  One wonders if any of
these non-Friedberg splits can be done uniformly.  It is also known
that for some c.e.\ sets the only nontrivial splits ($A = A_0 \sqcup
A_1$ and the $A_0$ and $A_1$ are not computable) are Friedberg. So one
cannot hope to get a uniform procedure which always provides a
nontrivial non-Friedberg split of every noncomputable c.e.\ set.  But
it would be nice to find a computable function $f(e) = \langle e_0,
e_1 \rangle$ such that, for all $e$, if $W_e$ is noncomputable then
$W_{e_0} \sqcup W_{e_1} = W_e$ is a nontrivial split of $W_e$ and, for
every c.e.\ set $A$, if $A$ has a nontrivial non-Friedberg split and
$A = W_e$ (so $W_e$ is any enumeration of $A$), and then $W_{e_0}
\sqcup W_{e_1} = W_e$ is a nontrivial non-Friedberg split.  So, if $A$
has a nontrivial non-Friedberg split and $W_e$ is any enumeration of
$A$, $f$ always gives out a nontrivial non-Friedberg split. In work
yet to appear, the author has shown that such a computable $f$ cannot
exist.

Let $\mathcal{P}$ be a property in $\mathcal{E}$.  We say that $A$ is
\emph{hemi}-$\mathcal{P}$ iff there are c.e.\ sets $B$ and $C$ such
that $A \sqcup B = C$ and $C$ has $\mathcal{P}$.  We can also define
\emph{Friedberg}-$\mathcal{P}$ iff there are c.e.\ sets $B$ and $C$
such that $A \sqcup B = C$ is a Friedberg split and $C$ has
$\mathcal{P}$. If $\mathcal{P}$ is definable then
\emph{hemi}-$\mathcal{P}$ and \emph{Friedberg}-$\mathcal{P}$ are also
definable.  One can get lots of mileage from the
\emph{hemi}-$\mathcal{P}$, see \citet{Downey.Stob:92} and
\citet{Downey.Stob:93}.  Most of these results are about properties
$\mathcal{P}$ where every nontrivial split of a set with $\mathcal{P}$
is Friedberg.  We feel that one should be using
\emph{Friedberg}-$\mathcal{P}$ rather than \emph{hemi}-$\mathcal{P}$.
To that end we ask the following:

 \begin{question}
   Is there a definable $\mathcal{P}$ such that the Friedberg splits
   are a proper subclass of the nontrivial splits?
 \end{question}

 We feel that the Friedberg splits are very special and they should
 not be able to always cover all the nontrivial splits of every
 definable property.


 \section{All orbits nice? No!}

 As we mentioned earlier, Friedberg also constructed a maximal set
 answering another question of Myhill.  A maximal set, $M$, is a c.e.\
 set such that for every superset $X$ either $X=^* M$ ($=^*$ is equal
 modulo finite) or $W =^* \omega$. Being maximal is definable.
 Friedberg's construction of a maximal set is very dynamic.
 \citet{Martin:66*1} showed that all maximal sets must be high.  A
 further result of \citet{Martin:66*1} shows a c.e.\ degree is high
 iff it contains a maximal set.  A remarkable result of Soare
 \cite{Soare:74} shows that the maximal sets form an orbit, even an
 orbit under automorphisms computable from $\bf 0''$ or
 $\Delta^0_3$-automorphisms.

 The result of Soare gives rise to the question are all orbits as nice
 as the orbit of the maximal sets?  We can go more into the formality
 of the question but that was dealt with already in another survey
 paper, \citet*{MR2395047}.  To tell if two c.e. sets, $A$ and $B$,
 are in the same orbit, it is enough to show if there is an
 automorphism $\Phi$ of $\mathcal{E}$ taking the one to the other,
 $\Phi(A) = B$ (we write this as $A$ is \emph{automorphic} to $B$).
 Hence it is $\Sigma^1_1$ to tell if two sets are in the same orbit.
 The following theorem says that is the best that we can do and hence
 not all orbits are as nice as the orbits of maximal sets.  The
 theorem has a number of interesting corollaries.

 \begin{theorem}[\citet*{MR2425182}]
   \label{sw} There is a c.e.\ set $A$ such that the index set $\{i :
   W_i \approx A\}$ is $\Sigma^1_1$-complete. \end{theorem}

\begin{corollary}[\citet{MR2425182}]
  Not all orbits are elementarily definable; there is no arithmetic
  description of all orbits of $\E$. \end{corollary}

\begin{corollary}[\citet{MR2425182}]
  The Scott rank of $\E$ is $\wock +1$. \end{corollary}

\begin{theorem}[\citet{MR2425182}]\label{sec:maincor}
  For all finite $\alpha > 8$ there is a properly $\Delta^0_{\alpha}$
  orbit. \end{theorem}

These results were completely explored in the survey,
\cite{MR2395047}.  So we will focus on some more recent work.  In the
work leading to the above theorems, Cholak and Harrington also showed
that:

\begin{theorem}[\cite{MR2366962}]
  Two simple sets are automorphic iff they are $\Delta^0_6$
  automorphic.  A set $A$ is \emph{simple} iff for every (c.e.) set
  $B$ if $A \cap B$ is empty then $B$ is finite.
\end{theorem}

Recently Harrington improved this result to show:

   \begin{theorem}[Harrington 2012, private email]
     The complexity of the $\mathcal{L}_{\omega_1,\omega}$ formula
     describing the orbit of any simple set is very low (close to 6).
   \end{theorem}

   That leads us to make the following conjecture:

   \begin{conjecture}
     We can build the above orbits in Theorem~\ref{sec:maincor} to
     have complexity close to $\alpha$ in terms of the
     $\mathcal{L}_{\omega_1,\omega}$ formula describing the orbit.
   \end{conjecture}

   \section{Complete Sets}

   Perhaps the biggest questions on the c.e. sets are the following:

 \begin{question}[Completeness]
   Which c.e.\ sets are automorphic to complete sets?
 \end{question}

 Motivation for this question dates back to Post.  Post was trying to
 use properties of the complement of a c.e.\ set to show that the set
 was not complete.  In the structure $\mathcal{E}$ all the sets in the
 same orbit have the same definable properties.

 By \citet{Harrington.Soare:98}, \cite{Harrington.Soare:91}, and
 \cite{Harrington.Soare:96}, we know that not every c.e.\ set is
 automorphic to a complete set and, furthermore, there is a dichotomy
 between the ``prompt'' sets and the ``tardy'' (nonprompt) sets with
 the ``prompt'' sets being automorphic to complete sets. We will
 explore this dichotomy in more detail, but more definitions are
 needed:


\begin{definition}
  $X = (W_{e_1} - W_{e_2}) \cup (W_{e_3} - W_{e_4}) \cup \ldots
  (W_{e_{2n-1}} - W_{e_{2n}})$ iff $X$ is $2n$-c.e.\ and $X$ is
  $2n+1$-c.e.\ iff $X = Y \cup W_e$, where $Y$ is $2n$-c.e.
\end{definition}

   \begin{definition}
     Let $X^n_e$ be the $e$th $n$-c.e.\ set.  $A$ is \emph{almost
       prompt} iff there is a computable nondecreasing function $p(s)$
     such that for all $e$ and $n$ $ \text{if } X^n_e = \overline{A}
     \text{ then } (\exists x) (\exists s) [ x \in X^n_{e,s} \text{
       and } x \in A_{p(s)}]$.
   \end{definition}

\begin{theorem}[\citet{Harrington.Soare:96}]
  Each almost prompt sets are automorphic to some complete set.
\end{theorem}

\begin{definition}
  $D$ is \emph{$2$-tardy} iff for every computable nondecreasing
  function $p(s)$ there is an $e$ such that $ X^2_e = \overline{D}$
  and $(\forall x) (\forall s) [\text{if }x \in X^2_{e,s}$ then $ x
  \not\in D_{p(s)}]$
\end{definition}

\begin{theorem}[\citet{Harrington.Soare:91}]
  There are $\E$ definable properties $Q(D)$ and $P(D,C)$ such that
  \begin{enumerate}
  \item $Q(D)$ implies that $D$ is $2$-tardy and hence the orbit of
    $D$ does not contain a complete set.
  \item for $D$, if there is a $C$ such that $P(D,C)$ and $D$ is
    $2$-tardy then $Q(D)$ (and $D$ is high).
  \end{enumerate}
\end{theorem}

The $2$-tardy sets are not almost prompt and the fact they are not
almost prompt is witnessed by $e=2$. It would be nice if the above
theorem implied that being $2$-tardy was definable. But it says with
an extra definable condition being $2$-tardy is definable.

\citet{Harrington.Soare:91} ask if each $3$-tardy set is computable by
some $2$-tardy set. They also ask if all low$_2$ simple sets are
almost prompt (this is the case if $A$ is low). With Gerdes and Lange,
Cholak answered these negatively:

\begin{theorem}[\citet*{MR2926283}]\label{sec:work-tardy-sets-1}
  There exists a properly $3$-tardy $B$ such that there is no
  $2$-tardy $A$ such that $B \leq_T A$. Moreover, $B$ can be built
  below any prompt degree. \end{theorem}

\begin{theorem}[\citet*{MR2926283}]
  There is a low$_2$, simple, $2$-tardy set. \end{theorem}

Moreover, with Gerdes and Lange, Cholak showed that there are
definable (first-order) properties $Q_n(A)$ such that if $Q_n(A)$
then $A$ is $n$-tardy and there is a properly $n$-tardy set $A$ such
that $Q_n(A)$ holds. Thus the collection of all c.e.\ sets not
automorphic to a complete set breaks up into infinitely many orbits.

But, even with the work above, the main question about completeness
and a few others remain open.  These open questions are of a more
degree-theoretic flavor. The main still open questions are:

\begin{question}[Completeness]
  Which c.e.\ sets are automorphic to complete sets? \end{question}

\begin{question}[Cone Avoidance]\label{sec:complete-sets}
  Given an incomplete c.e.\ degree $\mathbf{d}$ and an incomplete
  c.e.\ set $A$, is there an $\Ahat$ automorphic to $A$ such that
  $\mathbf{d} \not\leq_T \Ahat$? \end{question}

It is unclear whether these questions have concrete answers. Thus the
following seems reasonable.

\begin{question}
  Are these arithmetical questions? \end{question}

Let us consider how we might approach these questions. One possible
attempt would be to modify the proof of Theorem~\ref{sw} to add
degree-theoretic concerns. Since the coding comes from how $A$
interacts with the sets disjoint from it, we should have reasonable
degree-theoretic control over $A$. The best we have been able to do so
far is alter Theorem~\ref{sw} so that the set constructed has
hemimaximal degree and everything in its orbit also has hemimaximal
degree.  However, what is open is whether the orbit of any set
constructed via Theorem~\ref{sw} must contain a representative of
every hemimaximal degree or only have hemimaximal degrees. If the
infinite join of hemimaximal degrees is hemimaximal then the degrees
of the sets in these orbits only contain the hemimaximal degrees.  But,
it is open whether the infinite join of hemimaximal degrees is hemimaximal.


\subsection{Tardy Sets}

As mentioned above, there are some recent results on $n$-tardy and
very tardy sets (a set is very tardy iff it is not almost prompt). But
there are several open questions related to this work. For example, is
there a (first-order) property $Q_\infty$ so that if $Q_\infty(A)$
holds, then $A$ is very tardy (or $n$-tardy, for some $n$). Could we
define $Q_\infty$ such that $Q_n(A)\implies Q_\infty(A)$? How do
hemi-$Q$ and $Q_3$ compare? But the big open questions here are the
following:

\begin{question}
  Is the set $B$ constructed in Theorem~\ref{sec:work-tardy-sets-1}
  automorphic to a complete set? If not, does $Q_3(B)$ hold? 
\end{question}

It would be very interesting if both of the above questions have a
negative answer.

Not a lot about the degree theoretic properties of the $n$-tardies is
known. The main question here is whether
Theorem~\ref{sec:work-tardy-sets-1} can be improved to $n$ other than
$2$.

\begin{question}
  For which $n$ are there $n+1$ tardies which are not computed by
  $n$-tardies? \end{question}

But there are many other approachable questions. For example, how do
the following sets of degrees compare: 
\begin{itemize}
\item the hemimaximal degrees? 
\item the tardy degrees? 
\item for each $n$, $\{\mathbf{d} : $ there is an $n$-tardy $D$ such
  that $\mathbf{d} \leq_T D\}$? 
\item $\{\mathbf{d} : $ there is a $2$-tardy
  $D$ such that $Q(D)$ and $\mathbf{d} \leq_T D\}$?  
\item $\{\mathbf{d} : $ there is an $A \in \mathbf{d}$ which is not
  automorphic to a complete set$\}$?
\end{itemize}
Does every almost prompt set compute a $3$-tardy? Or a very tardy?
\citet{Harrington.Soare:98} show there is a maximal $2$-tardy set. So
there are $2$-tardy sets which are automorphic to complete sets. Is
there a nonhigh, nonhemimaximal, $2$-tardy set which is automorphic to
a complete set?

\subsection{Cone Avoidance, Question~\ref{sec:complete-sets}}

The above prompt vs.\ tardy dichotomy gives rise to a reasonable way
to address Question~\ref{sec:complete-sets}.  An old result of Cholak
\cite{mr95f:03064} and, independently, Harrington and Soare
\cite{Harrington.Soare:96}, says that every c.e.\ set is automorphic
to a high set.  Hence, a positive answer to both the following
questions would answer the cone avoidance question but not the
completeness question.  These questions seem reasonable as we know how
to work with high degrees and automorphisms, see \cite{mr95f:03064},

\begin{question}
  Let $A$ be incomplete.  If the orbit of $A$ contains a set of high
  prompt degree, must the orbit of $A$ contain a set from all high
  prompt degrees?
\end{question}

\begin{question}
  If the orbit of $A$ contains a set of high tardy degree, must the
  orbit of $A$ contain a set from all high tardy degrees?
\end{question}

Similarly we know how to work with prompt degrees and automorphisms,
see \citet*{mr92j:03039} and \citet{Harrington.Soare:96}.  We should
be able to combine the two. No one has yet explored how to work with
automorphisms and tardy degrees.

\section{$\D$-Maximal Sets}

In the above sections we have mentioned maximal and hemimaximal sets
several times.  It turns out that maximal and hemimaximal sets are
both $\mathcal{D}$-maximal.

\begin{definition}
  $\mathcal{D}(A) = \{ B : \exists W ( B \subseteq A \cup W \text{ and
  } W \cap A = \emptyset)\}$ under inclusion. Let
  $\E_{\mathcal{D}(A)}$ be $\E$ modulo $\mathcal{D}(A)$.
\end{definition}

$\mathcal{D}(A)$ is the ideal of c.e.\ sets of the form $\tilde{A}
\sqcup \tilde{D}$ where $\tilde{A} \subseteq A$ and $\tilde{D} \cap A
= \emptyset$.  

\begin{definition}
  $A$ is \emph{$\mathcal{D}$-hhsimple} iff $\E_{\mathcal{D}(A)}$ is a
  $\Sigma^0_3$ Boolean algebra.  $A$ is \emph{$\mathcal{D}$-maximal}
  iff $\E_{\mathcal{D}(A)}$ is the trivial Boolean algebra iff for all
  c.e.\ sets $B$ there is a c.e. set, $D$, disjoint from $A$, such
  that either $B \subset A \cup D$ or $B \cup D \cup A = \omega$.
\end{definition}
   
Maximal sets and hemimaximal sets are $\mathcal{D}$-maximal.  Plus,
there are many other examples of $\mathcal{D}$-maximal sets. In fact,
with the exception of the creative sets, all known elementary definable
orbits are orbits of $\mathcal{D}$-maximal sets.  In the lead up to
Theorem~\ref{sw}, Cholak and Harrington were able to show:

\begin{theorem}[\cite{mr2004f:03077}]
   If $A$ is $\mathcal{D}$-hhsimple and $A$ and $\Ahat$ are in the
   same orbit then $\E_{\mathcal{D}(A)} \cong_{\Delta^0_3}
   \E_{\mathcal{D}(\Ahat)}$.
\end{theorem}
 
So it is an arithmetic question to ask if the orbit of a
$\mathcal{D}$-maximal set contains a complete set.  But the question
remains does the orbit of every $\mathcal{D}$-maximal set contain a
complete set?  It was hoped that the structural properties of
$\mathcal{D}$-maximal sets would be sufficient to allow us to answer
this question.

Cholak, Gerdes, and Lange \cite{pub2} have completed a classification
of all $\D$-maximal sets. The idea is to look at how $\D(A)$ is
generated. For example, for a hemimaximal set $A_0$, $\D(A_0)$ is
generated by $A_1$, where $A_0 \sqcup A_1$ is maximal. There are ten
different ways that $\D(A)$ can be generated.  Seven were previously
known and all these orbits contain complete and incomplete sets.  Work
from \citet{MR1264963} shows that these seven types are not enough to
provide a characterization of all $\mathcal{D}$-maximal sets.  Cholak,
Gerdes, and Lange construct three more types and show that these ten
types provide a characterization of all $\mathcal{D}$-maximal sets.
We have constructed three new types of $\D$-maximal sets; for example,
a $\D$-maximal set where $\D(A)$ is generated by infinitely many not
disjoint c.e\ sets.  We show these three types plus another split into
infinitely many different orbits.  We can build examples of these sets
which are incomplete or complete. But, it is open if each such orbit
contains a complete set. So, the structural properties of
$\mathcal{D}$-maximal sets was not enough to determine if each
$\mathcal{D}$-maximal set is automorphic to a complete set.

It is possible that one could provide a similar characterization of
the $\mathcal{D}$-hhsimple sets. One should fix a $\Sigma^0_3$ Boolean
algebra, $\mathcal{B}$, and characterize the $\mathcal{D}$-hhsimple
sets, $A$, where $\E_{\mathcal{D}(A)} \cong \mathcal{B}$.  It would be
surprising if, for some $\mathcal{B}$, the characterization would
allow us to determine if every orbit of these sets contains a complete
set.

\section{Lowness}

Following his result that the maximal sets form an orbit,
Soare\cite{Soare:82} showed that the low sets resemble computable
sets.  A set $A$ is low$_n$ iff $\mathbf{0}^{(n)} \equiv_T
A^{(n)}$. We know that noncomputable low sets cannot have a computable
set in their orbit, so, the best that Soare was able to do is the
following:

\begin{definition}
  $\mathcal{L}(A)$ are the c.e.\ supersets of $A$ under inclusion.
  $\mathcal{F}$ is the filter of finite sets.  $\mathcal{L}^*(A)$ is
  $\mathcal{L}(A)$ modulo $\mathcal{F}$.
\end{definition}

\begin{theorem}[\citet{Soare:82}]
  If $A$ is low then $\mathcal{L}^*(A) \approx
  \mathcal{L}^*(\emptyset)$.
\end{theorem}

In 1990, Soare conjectured that this can be improved to low$_2$. Since
then there have been a number of related results but this conjecture
remains open.  To move forward some definitions are needed:

\begin{definition}
  $A$ is \emph{semilow} iff $\{ i | W_i \cap A \neq \emptyset \}$ is
  computable from $\mathbf{0'}$.  $A$ is \emph{semilow$_{1.5}$} iff
  $\{ i | W_i \cap A \text{ is finite}\} \leq_1 \mathbf{0''}$.  $A$ is
  \emph{semilow$_2$} iff $\{ i | W_i \cap A \text{ is finite}\}$ is
  computable from $\mathbf{0''}$.
\end{definition}

Semilow implies semilow$_{1.5}$ implies semilow$_2$, if $A$ is low
then $\overline{A}$ is semilow, and low$_2$ implies semilow$_2$
(details can be found in \citet{Maass:83} and
\citet{mr95f:03064}). \citet{Soare:82} actually showed that if
$\overline{A}$ is semilow then $\mathcal{L}^*(A) \approx
\mathcal{L}^*(\emptyset)$.  \citet{Maass:83} improved this to when
$\overline{A}$ is semilow$_{1.5}$.

In Maass's proof semilow$_{1.5}$ness is used in two ways:  A c.e.\
set, $W$, is \emph{well-resided outside $A$} iff $W \cap \overline{A}$
is infinite.  Semilow$_{1.5}$ makes determining which sets are
well-resided outside $A$ a $\Pi^0_2$ question.  The second use of
semilow$_{1.5}$ was to capture finitely many elements of $W \cap
\overline{A}$.  For that Maass showed that semilow$_{1.5}$ implies the
\emph{outer splitting property}:

\begin{definition}
  $A$ has the \emph{outer splitting property} iff there are computable
  functions $f, h$ such that, for all $e$, $W_e = W_{f(e)} \sqcup
  W_{h(e)}$, $W_{f(e)} \cap \overline{A}$ is finite, and if $W_e \cap
  \overline{A}$ is infinite then $W_{f(e)} \cap \overline{A}$ is
  nonempty.
\end{definition}

Cholak used these ideas to show that:

\begin{theorem}[\citet{mr95f:03064})]
  If $A$ has the outer splitting property and $\overline{A}$ is
  semilow$_2$ then $\mathcal{L}^*(A) \approx
  \mathcal{L}^*(\emptyset)$.
\end{theorem}

It is known that there is a low$_2$ set which does not have the outer
splitting property, see \citet*[Theorem~4.6]{Downey:2013wq}.  So to
prove that if $A$ is low$_2$ then $\mathcal{L}^*(A) \approx
\mathcal{L}^*(\emptyset)$ will need a different technique.  However,
\citet{Lachlan:68} showed that every low$_2$ set has a maximal
superset using the technique of \emph{true stages}.  Perhaps the
true stages technique can be used to show Soare's conjecture.

Recently there has been a result of Epstein. 

 \begin{theorem}[\citet{epstein:10} and \cite{MR3003266}]
   There is a properly low$_2$ degree $\mathbf{d}$ such that if $A
   \leq_T \mathbf{d}$ then $A$ is automorphic to a low set.
 \end{theorem}

 Epstein's result shows that there is no collection of c.e.\ sets
 which is closed under automorphisms and contains at least one set of
 every nonlow degree. Related results were discussed in
 \citet{mr2001k:03085}. 

 This theorem does have a nice yet unmentioned corollary: The
 collection of all sets $A$ such that $\overline{A}$ is semilow (these
 sets are called \emph{speedable}) is not definable.  By
 \citet*[Theorem~4.5]{Downey:2013wq}, every nonlow c.e. degree
 contains a set $A$ such that $\overline{A}$ is not semilow$_{1.5}$
 and hence not semilow. So there is such a set $A$ in $\mathbf{d}$.
 $A$ is automorphic to a low set $\Ahat$.  Since $\Ahat$ is low,
 $\overline{\Ahat}$ is semilow.

 Esptein's result leads us wonder if the above results can be improved
 as follows:

 \begin{conjecture}[Soare]
   Every semilow set is (effectively) automorphic to a low set.
 \end{conjecture}

 \begin{conjecture}[Cholak and Epstein]
   Every set $A$ such that $A$ has the outer splitting property and
   $\overline{A}$ is semilow$_2$ is automorphic to a low$_2$ set.
 \end{conjecture}

 Cholak and Epstein are currently working on a proof of the latter
 conjecture and some related results.  Hopefully, a draft will be
 available soon.

\bibliographystyle{plainnat} 

\bibliography{incomputable}

\end{document}